\documentclass[12pt]{amsart}
\usepackage{amscd,amsthm,amsfonts,amssymb,amsmath,euscript}
\usepackage[dvips]{graphicx}
\usepackage[dvips]{graphics}
\usepackage[T2A]{fontenc}
\usepackage[cp866]{inputenc}

\newtheorem{definition}{Definitions}[section]
\newtheorem{theorem}{Theorem}[section]
\newtheorem{lemma}{Lemma}[section]
\newtheorem{example}{Example}[section]
\newtheorem{proposition}{Proposition}[section]

\newcommand{\subgroups}{\mathop{\sf Sub}\nolimits}
\newcommand{\rwr}{\mathop{\sf wr}\nolimits}

\title{Algebras of conjugacy classes of partial elements}
\date{}
\author
{Andrei V. Alexeevski, \qquad Sergey M. Natanzon}

\begin{document}

\dedicatory{Dedicated to Sergey Petrovich Novicov on the occasion of his 75th birthday}

\maketitle

\section*{Abstract}

In 2001 Ivanov and Kerov associated with the infinite permutation group $S_\infty$ certain commutative associative algebra $A_\infty$ called the algebra
of conjugacy classes of partial elements.
A standard basis of $A_\infty$ is labeled by Yang diagrams of all orders.

Mironov, Morozov, Natanzon, 2012, have proved that the completion of $A_\infty$ is isomorphic to the direct product of centers of group algebras of groups $S_n$. This isomorphism was explored in a construction of infinite dimensional Cardy-Frobenius algebra corresponding to
 asymptotic Hurwitz numbers.

In this work algebras of conjugacy classes of partial elements are defined for a wider class of infinite groups. It is proven that completion of any such algebra is isomorphic to the direct product of centers of group algebras of relevant subgroups.

\section*{Introduction}

Commutative Frobenius algebras with fixed linear functionals
are important for mathematical physics, since they 1-1 correspond to 2D closed topological field theories \cite{D1}.

Hurwitz numbers of degree $n$ generates a Frobenius algebra and linear functional on it. This algebra is isomorphic
to  the center $Z(k[S_n])$ of the group algebra of the permutation group $S_n$ over a field $k$. Therefore, Hurwitz numbers correspond
to certain 2D closed topological field theory \cite{D2}.
Classical Hurwitz theory and particularly  Hurwitz numbers were generalized to the coverings over surfaces with boundaries \cite{AN2}. These generalized Hurwitz numbers  correspond to certain open-closed topological field theory and to  more general Klein topological field theory \cite{AN,AN1,AN2}.

Classical Hurwitz numbers of all degrees lie in the base of the construction of infinite dimensional  2D closed topological field theory \cite{MMN4}. Corresponding
Frobenius algebra is algebra $A_{\infty}$ introduced by Ivanov and Kerov \cite{IK} in conjunction with studying the group $S_\infty$ of
finitary permutations of the set of natural numbers. We call $A_\infty$ the IK-algebra.

Definition of IK-algebra $A_{\infty}$  is based on the multiplication of
'partial permutations', i.e. pairs (d,s) consisting of a subset $d$ of the set of natural numbers $\mathbb{N}$ and a permutation $s$ acting on $d$ and trivially on $\mathbb{N}\setminus d$. The product is defined by formula
$$(d',s')\circ  (d'',s'') = ( d'\cup d'', s'\circ s'').$$
The basis of $A_{\infty}$ is formed by sums of elements of a conjugacy class of partial elements. (A conjugacy class
is  an orbit of $S_\infty$ action on partial elements.) Basic elements are in 1-1 correspondence with Yang diagrams of all orders.

IK-algebra is isomorphic to the algebra of shifted Schur functions \cite{OO}, and to the algebra of cut-and-join operators  in the frame of the theory of asymptotic Hurwitz numbers \cite{MMN1,MMN2}.

Completion of $A_\infty$ is isomorphic to the direct product $\prod\limits_{n} Z(k[S_n])$
of the centers of group algebras of $S_n$   \cite{MMN3,MMN4}.

\vspace{2ex}

In this work we define algebras of conjugacy classes of partial elements for a wider class of infinite groups and prove that completion of any such algebra is isomorphic to the direct product of centers of group algebras of relevant subgroups.

Instead of infinite permutation group $S_\infty$
we consider a group $G$ acting by automorphisms of a poset (partial ordered set) $\Lambda$. In the case of infinite permutation group
 $\Lambda$ is the poset of all finite subsets of the set of natural numbers and $\Lambda/G=\{0\}\cup\mathbb{N}$.

Representation $\eta$ of the poset $\Lambda$ in the poset of finite
subgroups of $G$ is fixed. Coincidence of the intersection of a conjugacy class $c$ of partial elements of G  with a conjugacy class of partial elements in a subgroup $G_\lambda=\eta(\lambda)$ is
an obligatory condition on $(\Lambda,G,\eta)$. This condition establishes relations between conjugacy classes in subgroups $G_\lambda$ and group $G$.

If the set of data $(\Lambda,G,\eta)$ is given then a pair $(\lambda,h)$ consisting of $\lambda\in \Lambda$ and $h\in G_\lambda$  is called a  partial element.
Algebra $A=A(\Lambda,G,\eta)$ is defined as linear envelope of formal sums of elements of a conjugacy class of partial elements. We call $A$ a generalized IK-algebra.
We prove that a completion of $A$ is isomorphic to
$\prod\limits_{l \in\Lambda/G} Z(k[G_{\lambda}])$, where $\lambda$ is a representative of the orbit $l\in\Lambda/G$ and $Z(k[G_{\lambda}])$ is the center of the group algebra of $G_{\lambda}$. We give also exact formulas connecting structural constants of algebras $A$ and
$\prod\limits_{l\in\Lambda/G} Z(k[G_{\lambda}])$.

\vspace{2ex}

In section 1 we provide axioms of an admissible family of subgroups of a group $G$. The most restrictive axiom concerns intersections of conjugacy classes of $G$ with subgroups $G_\lambda$ of the family. In section 2 we present a series of examples of admissible families of subgroups. All  examples are constructed as restricted wreath products of a finite group $F$ and the group $S_\infty$. Among this series of examples there is an infinite Weyl group
of type $B_\infty$.  In all cases $\Lambda$ is the poset  of finite subsets
of $\mathbb{N}$. Nevertheless, keeping in mind putative other examples, we develop the theory for arbitrary posets $\Lambda$ satisfying axioms.

In section 3 we define  a generalized IK-algebra $A=A(\{G_{\lambda}\})$ associated with an admissible family of subgrpoups.
Algebra $A$ is associative and commutative. We describe its structure constants.
In section 4 we prove, that $A$ is a projective limit of generalized IK-algebras $A_{\preceq\lambda}$ corresponding to admissible subfamilies
$\{G_\mu | \mu\preceq\lambda\}$ of groups $G_\lambda$.
In section 5 we express the structure constants of $A$ via structure constants of $\prod\limits_{l\in\Lambda/G} Z(k[G_{\lambda}])$ and vise versa.
In section 6 we construct a monomorphism $\varphi$ from $A$ to $\prod\limits_{l \in \Lambda/G} Z(k[G_{\lambda}])$.
In section 7 we define a completion $\bar A$ of algebra $A$ and prove that continuation of $\varphi$ is the isomorphism of $\bar A$ and $\prod\limits_{l \in \Lambda/G} Z(k[G_{\lambda}])$.
\vspace{2ex}


\section{Admissible family of subgroups}

We call poset (partially ordered set) $\Lambda$ \textit{admissible} if the following conditions hold:
\begin{itemize}
\item $\Lambda$ has minimal element $0$,
\item for each $\lambda$ the set $\{\lambda' |\lambda'\preceq\lambda\}$ is finite,
\item each finite subset $F\subset\Lambda$ has a supremum $^\wedge F\in\Lambda$.
\end{itemize}

The supremum $\lambda'\wedge \lambda''$ of two elements of an admissible poset $\Lambda$ defines a multiplication on $\Lambda$.
This multiplication is commutative and associative because $(\lambda'\wedge\lambda'')\wedge\lambda'''=\lambda'\wedge(\lambda''\wedge\lambda''')={}^\wedge\{\lambda',\lambda'',\lambda'''\}$. Thus, $(\Lambda,\wedge)$ is a commutative semigroup with the unit $0$.

In examples below $\Lambda$ is a set of finite subsets of natural numbers $\mathbb{N}$, minimal element is the empty set and a partial ordering is the inclusion of subsets.

Let $G$ be a group acting on  $\Lambda$ by automorphisms, i.e. by transforms preserving partial order $\preceq$. Then $G$ acts also on  the semigroup $(\Lambda, \wedge)$ by  isomorphisms.
Denote by $\subgroups{G}$ the poset of subgroups of group $G$ with partial ordering being inclusion of subgroups and with action of $G$ on  $\subgroups{G}$ by conjugations. Let $\eta:\Lambda\to\subgroups{G}$ be a morphism of posets compatible with the action of $G$.  Denote the image of $\lambda\in\Lambda$ by  $G_\lambda\in\subgroups{G}$. Note, that subgroups $G_{\lambda'}$ and $G_{\lambda''}$ may coincide even if  $\lambda'\ne\lambda''$.

The set $L=\Lambda/G$ of $G$-orbits in admissible poset $\Lambda$ inherits partial ordering : $l'\in L$  precedes $l''\in L$ if there are elements $\lambda'\in l$ and $\lambda''\in l''$ such that $\lambda' \preceq \lambda''$.
For example, if  $\Lambda$ is a set of finite subsets of natural numbers, then $L$ is the ordered set  $L=\{0,1,2,\dots\}$.

Evidently, the element $0$ form an orbit of $G$ and thus, $0$ is minimal element of $L$. It is clear also, that for each $l\in L$ there are
finitely many $l'\in L$ such that $l'\preceq l$.

Denote by $l'\wedge l''$ the set of orbits of all elements $\lambda'\wedge\lambda''$, where $\lambda'\in l'$, $\lambda''\in l''$. Clearly, if
$\lambda'\wedge\lambda''\in l$ then all other elements of the orbit $l$ also are $\wedge$-products of elements of $l'$ and $l''$.

Fix $l\in L$. All subgroups $G_\lambda$, $\lambda\in l$ are conjugated in $G$. Denote this conjugacy class of subgroups by $G_l$ .
A representative of this conjugacy class we will also denote by $G_l$ unless this  leads to a confusion.

Let $\lambda$ be an element of $\Lambda$ and $h$ be an element of the subgroup $G_\lambda$. Then,
following \cite{IK}, we call pair $(\lambda,h)$  \textit{a partial element} of group $G$.

Group $G$ acts on the set of partial elements: $(\lambda,h)\to (g\lambda,ghg^{-1})$ for $g\in G$. We call this action \textit{the conjugation of partial elements} because $G_{g\lambda}=gG_\lambda g^{-1}$. An orbit $G(\lambda,h)$ of a partial element is called \textit{a conjugacy class of partial elements}.

\begin{definition} \label{def.admissible} A triple $(\Lambda,G{\colon}\!\Lambda,\eta:\Lambda\to \subgroups G)$  consisting of an admissible poset $\Lambda$, a group $G$ acting by automorphisms of $\Lambda$ and a $G$-compatible morphism $\eta$ of posets  is called admissible set of data, and the image $\eta(\Lambda)\subset \subgroups G$ is called an admissible family of subgroups, if
\begin{enumerate}
\item $G_0=\{e\}$,
\item $G_\lambda$ is finite subgroup for each $\lambda\in\Lambda$,
\item  $\wedge$-product $l'\wedge l''$ of any two $G$-orbits $l', l''\in L=\Lambda/G$ is a finite set
\item if  $l'\wedge x$ is equal to $l''\wedge x$ for all $x\in L$ and $G_{l'}=G_{l''}$ then $l'=l''$
\item \label{five}  for any $\lambda', \lambda'' \preceq \lambda$ and any two partial elements $(\lambda',h'), (\lambda'',h'')$ conjugated in group $G$, they are also conjugated in $G_\lambda$.
\end{enumerate}
\end{definition}

Condition \ref{five} is restrictive one. It implies, for example,  that for each $\lambda\in\Lambda$ the factor-group
 $\mathcal{N}_G(G_\lambda)/G_\lambda$  the normalizer acts on the set of conjugacy classes of $G$ trivially.

Writing \textit{'an admissible family of subgroups $\eta(\Lambda)\subset\subgroups{G}$} we always assume that
 $(\Lambda,G,\eta)$ is an admissible set of data.

Let $\hat G$ be a join of all subgroups $G_\lambda$, $\lambda\in \Lambda$. The set $\hat G$ is a subgroup because $G_{\lambda'}, G_{\lambda''}\subset G_{\lambda'\wedge\lambda''}$ for any $\lambda', \lambda''\in \Lambda$. Clearly, $\hat G$ is a  normal subgroup and the triple $(\Lambda,\eta,\hat G)$ generates the same admissible family of subgroups as $(\Lambda,\eta,G)$ does. Below we assume additionally that $G=\hat G$.

Note that group $G$ may be finite or infinite.

\section{Examples of admissible families of subgroups}\label{examples} \label{s.examples}
\begin{example}
Admissible family of subgroups of $S_\infty$.
\end{example}

Let $G=S_\infty$  be the group of all finitary permutations of the set of natural numbers $\mathbb{N}$  and $\Lambda$ be the  poset of all finite subsets of $\mathbb{N}$. Evidently, $\Lambda$ is an admissible poset with empty set being its minimal element.  $S_\infty$ acts on $\mathbb{N}$ and therefore acts on $\Lambda$. Define the morphism $\eta:\Lambda\to\subgroups{G}$ by setting $\eta(\lambda)=S_\lambda$
where $S_\lambda$ denotes the subgroup
of all permutations acting trivially on $\mathbb{N}\setminus\lambda$.

Note that $S_\emptyset$ is equal to $\{e\}$  as well as all subgroups $S_{\{i\}}$ for $i\in\mathbb{N}$; all they are images of different elements of $\Lambda$. No other coincidences are in the set $\eta(\Lambda)$.

\begin{theorem}\label{ex1}
The set of data $(\Lambda,S_\infty,\eta)$ is admissible and $\eta(\Lambda)$ is admissible family of subgroups.
\end{theorem}

\begin{proof} In fact, it was proven in \cite{IK}.  In our axiomatization (see definition \ref{def.admissible} ) we should check that if
two partial elements $(\lambda',h')$ and $(\lambda'',h'')$ belong to $S_\lambda$ (i.e. $\lambda',\lambda''\preceq\lambda$ and therefore
$S_{\lambda'},S_{\lambda''}\subset S_\lambda$)  and they are conjugated in $S_\infty$ then they are conjugated in $S_\lambda$.
Indeed, conjugation in $S_\infty$ implies that $|\lambda'|=|\lambda''|$ and clearly, two subsets of $\lambda$ with equal cardinality can be superposed by a permutation from $S_\lambda$. We may assume that $\lambda'=\lambda''$. The conjugation in $G$ means also coincidence of cyclic types (equivalently, Young diagrams) of permutations $h'$ and $h''$. Hence, $h'$ and $h''$ are conjugated in the $S_\lambda$.
\end{proof}

\begin{example}
Admissible family of subgroups of  the  group $F\rwr S_\infty$
\end{example}
Here by $G=F\rwr S_\infty$ is denoted a restricted wreath product of finite group $F$ and the group $S_\infty$ of finitary permutations of $\mathbb{N}$. By definition, an element of $G$ is a sequence $(s;f_1,f_2,\dots)$ of elements $s\in S_\infty$ and $f_i\in F$ such that all but
 finitely many  $f_i$ are identity elements. Product of two
elements $h'= (s';f_1',f_2',\dots)$ and $h''= (s'';f_1'',f_2'',\dots)$ is $h'h''= (s's''; f_{s''^{-1}(1)}, f_{s''^{-1}(2)}, \dots)$.

Note that $G$ is semidirect product of $S_\infty$ and $F_\infty=F\times F \dots$ where $F_\infty$ denotes the direct product of copies of the group $F$ marked by natural numbers with finitary condition: $(f_1,f_2,\dots)\in F_\infty$ implies that  finitely many components $f_i\ne e$.

Let $\Lambda$ be the set of all finite subsets of $\mathbb{N}$.  Group $G$ acts on $\mathbb{N}$ and therefore $G$ acts on $\Lambda$. The kernel of this action is $F_\infty$.

We call a subset $\varsigma\in\Lambda$ the support of an element
$h= (s;f_1,f_2,\dots)\in G$ if for  each $n\in\varsigma$ either $s(n)\ne n$ or $f_n\ne e$  and for each $m\notin \lambda$ it is true that  $s(m)=m$ and $f_m=e$.

Define $\eta:\Lambda\to\subgroups{G}$ by setting $\eta(\lambda)=G_\lambda$ where $G_\lambda$ is the subgroup of all elements such that their support $\varsigma$ is contained in $\lambda$.

\begin{theorem}\label{ex1}
The set of data $(\Lambda,F\rwr S_\infty,\eta)$ is admissible and $\eta(\Lambda)$ is an admissible family of subgroups.
\end{theorem}

\begin{proof} The poset $\Lambda$ is the same as in previous example and hence is admissible. It is sufficient to prove condition \ref{five} of the definition \ref{def.admissible}.  Let $(\lambda',h')$ and $(\lambda'',h'')$ be two partial elements such that both of them belong to
$G_\lambda$ (i.e. $\lambda',\lambda''\preceq\lambda$ and therefore $G_{\lambda'},G_{\lambda''}\subset G_\lambda$)  and they are conjugated in $G$. We should prove that they are conjugated in $G_\lambda$.
Indeed, conjugation in $G$ implies that $|\lambda'|=|\lambda''|$ and clearly, two subsets of $\lambda$ with equal cardinality can be superposed by a permutation from $S_\lambda=G_\lambda/F_\lambda$ where $F_\lambda$ is the subgroup of $F_\infty$ consisting of all elements with the support in $\lambda$. Thus, we may assume that $\lambda'=\lambda''$. Denote by $\varsigma'$ (resp., $\varsigma''$) the support of the element $h'$ (resp., $h''$). Clearly, $\varsigma', \varsigma'' \preceq \lambda$ (generally, $\varsigma', \varsigma''$ may not be equal to $\lambda$). We are given that the partial
elements are conjugated in $G$, therefore $|\varsigma'|= |\varsigma''|$. Hence there is permutation in $G_\lambda$ that superpose  $\varsigma'$ and $\varsigma''$. Thus, we may assume $\varsigma'= \varsigma''$. Evidently, the element $g\in G$ such that $g(\lambda',h')=(\lambda'',h'')$ must preserve
$\varsigma'$. Obviously, there is element $g'\in G_{\varsigma'}$ such that $g'$-action on $G_{\varsigma'}$ coincides with the action of $g$ on $G_{\varsigma'}$. Note that $G_{\varsigma'}\subset G_\lambda$ and we proved conjugation of the given partial elements in the $G_\lambda$.
\end{proof}

\vspace{2ex}

One of partial cases of the restricted wreath product  is  $B_\infty$, infinite dimensional group of finitary  automorphisms of the root system
$\Sigma(B_\infty) = \{ \pm e_i, \pm e_i\pm e_j | i, j \in \mathbb{N}, i\ne j\}$ where $\{e_i | i\in \mathbb{N}\}$ is fixed orthogonal basis of
Euclidean space $\mathbb{R}^\infty$. In this case $F=\{\pm 1\}$ is two-element group. The group $G=F\rwr S_\infty$ is the Weyl group for root system $\Sigma(B_\infty)$.

It is of interest to test other classical root systems and corresponding Weyl groups as a sources of admissible families of subgroups. Weyl group of type $A_\infty$ is isomorphic to $S_\infty$, see example 1 (we have to suppose, that Weyl group of type $A_0$ is equal to $\{e\}$) .  Weyl group of type $C_\infty$ is isomorphic to the Weyl groups of type $B_\infty$. In the case of type $D_\infty$ there is an obstacle to define an admissible family of subgroups. Indeed, there are pairs of different conjugacy classes of the Weyl group of type $D_{2n}$ that belong to one conjugacy class of $D_m$, $m>2n$ (see, for example, \cite{Carter}). Therefore,  condition \ref{five} of the definition  \ref{def.admissible} is not satisfied.


\section{Algebra of conjugacy classes of partial elements}

Let $\eta(\Lambda)$ be an admissible family of subgroups of a group $G$. Denote by $E=E(\Lambda,G,\eta)$ the set of all partial elements.
Define multiplication on $E(\Lambda,G,\eta)$  by formula
$$(\lambda',h')(\lambda'',h'') = (\lambda'\wedge\lambda'',h'h'')$$
Clearly, $h'h''$ is an element of the subgroup $G_{\lambda'\wedge\lambda''}$, thus $(\lambda'\wedge\lambda'',h'h'')$ is a partial element. The multiplication is associative since the multiplication $\lambda'\wedge\lambda''$ in poset $\Lambda$ is associative.

Denote by $\Omega$ the set of conjugacy classes of partial elements (i.e. the set of $G$-orbits on the set of all partial elements).

Let $(\lambda,h)$ be a partial element. Denote by $l$ the orbit of $\lambda\in\Lambda$ and by $c$ the conjugacy class of $h$ in the group $G$.

\begin{lemma}  For fixed $G$-orbit $l\in L$ and conjugacy class $c$ of group $G$ the subset of partial elements  $\{(\lambda,h) | \lambda\in l, h\in c\}$ is either empty or coincides with a conjugacy class of partial elements.
 \end{lemma}
\begin{proof} By definition, a pair $(\lambda,h)$ is a partial element if and only if $h\in G_\lambda$.  Lemma follows immediately from the  definition \ref{def.admissible}
\end{proof}

Thus, a conjugacy class of partial elements $\omega\in\Omega$ we may (and will) denote by $\omega=(l,c)$, where $l\in L$ and  $c$ is a conjugacy class of group $G$.

We denote by $e_\omega$ the formal sum of all partial elements of a conjugacy class $\omega\in\Omega$.  If
$\omega=(l,c)$ then we also use denotation $e_{(l,c)}$; if $(l,c)=\emptyset$ then we put $e_{(l,c)}=0$.

Let $A= A(\Lambda,G,\eta)$ be the linear envelope of all elements $e_\omega$, $\omega  \in\Omega$ over a field $k$.


\begin{theorem} The multiplication on partial elements induces a structure of associative commutative algebra on $A(\Lambda,G,\eta)$.
\end{theorem}

\begin{proof} Although elements $e_{(l,c)}$ may be infinite sums, their products are defined correctly.
Indeed, for a fixed partial element $(\lambda,h)$ there are only finitely many pairs of partial elements  $(\lambda',h')\in (l',c')$, $(\lambda'',h'')\in (l'',c'')$ such that $(\lambda'\wedge\lambda'', h'h'')= (\lambda,h)$  because there are finitely many $\tilde{\lambda}\in\Lambda$ such that
$\tilde{\lambda}\preceq\lambda$, and group $G_\lambda$ is finite.

Let the product $e_{(l',c')} e_{(l'',c'')}$ is equal to $\sum_{(\lambda,h)}\gamma_{(\lambda,h)}e_{(\lambda,h)}$. Clearly,
coefficients $\gamma_{(\lambda,h)}$ are equal for conjugated partial elements and only finitely many conjugated classes of partial elements appear in the product according to definition \ref{def.admissible}.

The associativity of $A(\Lambda,G,\eta)$ follows from associativity of the multiplication of partial elements. The commutativity of  $A(\Lambda,G,\eta)$ is evident because elements $e_\omega$ are invariant under $G$-action by conjugation of partial elements.
\end{proof}

We call algebra $A=A(\Lambda,G,\eta)$ \textit{an algebra of conjugacy classes of partial elements} or \textit{a generalized IK-algebra} ('IK' is for Ivanov, Kerov, see \cite{IK}). Actually, $A$ is the center of semigroup algebra of the semigroup of partial elements.

Algebra $A$ has natural basis $\{e_\omega | \omega\in \Omega\}$. Denote structure constants of $A$ in this basis by
$P_{\omega',\omega''}^\omega$.

\begin{lemma} \label{l.struct_const} Let $(\lambda,h)$ be an element of a conjugacy class of partial elements $\omega = (l,c)$. Then
\begin{equation*}
\begin{split}
&P_{\omega',\omega''}^{\omega}=\\
&=|\{ (\lambda',h'),(\lambda'',h'')| (\lambda',h')\in\omega',
(\lambda'',h'')\in\omega'', \lambda'\wedge\lambda'' = \lambda, h'h''=h\}|
\end{split}
\end{equation*}
\end{lemma}
\begin{proof} The proof follows directly from the definitions. \end{proof}



\section{Inverse (projective) limit of subalgebras of conjugacy classes}
Let $\eta(\Lambda)$ be an admissible family of subgroups of a group $G$ and $A=A(\Lambda,G,\eta)$ be the algebra of  conjugacy classes of partial elements. Fix an element $\lambda\in\Lambda$.
 Denote by $\Lambda_{\preceq\lambda}$ the poset of all elements preceding $\lambda$, and by $\eta_\lambda$ the restriction of $\eta$ onto $\Lambda_{\preceq\lambda}$. Clearly $\eta_\lambda(\Lambda_{\preceq\lambda})$ is an admissible family of subgroups of $G_\lambda$. The algebra of conjugacy classes of  partial elements $A_{\preceq\lambda}=A(\Lambda_{\preceq\lambda}, G_\lambda,\eta_\lambda)$ is finite-dimensional because
 $|\Lambda_{\preceq\lambda}|<\infty$ and group $G_\lambda$ is finite.

Conjugacy classes of partial elements of $E(\Lambda_{\preceq\lambda}, G_\lambda,\eta_\lambda)$ are equal to intersections of conjugacy classes of partial elements in $E(\Lambda, G,\eta)$ with $E(\Lambda_{\preceq\lambda}, G_\lambda,\eta_\lambda)$ (it follows from the definition \ref{def.admissible}). We denote these intersections by $(l,c)_\lambda$ where $l\in L=\Lambda/G$ and $c$ is a conjugacy class in $G$. If conjugacy class $(l,c)$ does not intersects with $(\Lambda_{\preceq\lambda}, G_\lambda)$ then we put $(l,c)_\lambda=\emptyset$.

According to lemma \ref{l.struct_const}, if orbits $\omega,\omega',\omega''$ intersect with $G_\lambda$ then   the structure constant $P_{\omega',\omega''}^{\omega}$ coincides with structure constant
$P_{\omega',\omega''}^{\omega}(\lambda)$ of algebra $A_{\preceq\lambda}$.

For any pair of element  $\lambda' \preceq \lambda''\in \Lambda$  define a linear map
$\pi_{\lambda',\lambda''}:A_{\preceq\lambda''}\to A_{\preceq\lambda'}$ by the formula
$$
\pi_{\lambda',\lambda''}(e_{(l,c)_{\lambda''}}) = (e_{(l,c)_{\lambda'})})
$$
As above, we assume that $e_{(l,c)_{\lambda''}}=0$ if $(l,c)_{\lambda''}=\emptyset$.

\begin{lemma} \label{t3.1} Linear map $\pi_{\lambda',\lambda''}$ is an epimorphism of algebras.
\end{lemma}

\begin{proof} Lemma \ref{l.struct_const}  provides the equality of structure constants for those elements of the basis of $A_{\preceq\lambda''}$ that are mapped not to $0$. Clearly, the kernel of $\pi_{\lambda',\lambda''}$ is an ideal.
\end{proof}

Algebras $A_{\preceq\lambda}$ and epimorphisms $\pi_{\lambda',\lambda''}$ form projective system of associative commutative finite-dimensional algebras with respect to poset $\Lambda$.

\begin{theorem} \label{t3.1} Inverse limit $ \underleftarrow{\lim} A_\lambda$ is isomorphic to the algebra
$A=A(\Lambda,G,\eta)$ of conjugacy classes of partial elements.
\end{theorem}
\begin{proof} Define epimorphisms $\pi_{\lambda}: A\to A_{\preceq\lambda}$ by the same formula as for $\pi_{\lambda',\lambda''}$.
Clearly, $\pi_{\lambda',\lambda}\circ\pi_{\lambda} = \pi_{\lambda'}$. The minimality of $A$ among algebras with the same morphisms is  evident.
\end{proof}

If $\lambda',\lambda''\in\Lambda$ belong to the same $G$-orbit, then there is canonical isomorphisms between algebras $A_{\preceq\lambda'}$ and $A_{\preceq\lambda''}$, defined on the bases by formula: $e_{(l,c)_{\lambda'}}\to e_{(l,c)_{\lambda''}}$. These isomorphisms allows to identify all algebras $A_{\preceq\lambda}$ with $\lambda$ from the same $G$-orbit $l\in\Lambda/G$.
We denote a representative of these class of canonically isomorphic algebras by $A_{\preceq l}$.


\section{Relations between structure constants of generalized IK-algebras and centers of group algebras}

By definition \ref{def.admissible}, the intersection of a  conjugacy $c$ of the group $G$ with a subgroup $G_\lambda$ is either empty or
 coincides with a conjugacy class of $G_\lambda$. In the latter case we denote this intersection by $c(\lambda)$.

Sums $e_{c(\lambda)}$  of elements of the nonempty intersection of a conjugacy class of $G$ with $G_\lambda$ form a basis of the center $Z(k[G_\lambda])$ of the group algebra of $G_\lambda$. Structure constants of $Z(k[G_\lambda])$ we denote by
$S_{c'(\lambda),c''(\lambda)}^{c(\lambda)}$. Put  $S_{c'(\lambda),c''(\lambda)}^{c(\lambda)}=0$ if any of conjugacy classes $c, c', c''$ does not intersect with $G_\lambda$.

All algebras $Z(k[G_\lambda])$ with $\lambda$ from the same $G$-orbit $l$ are canonically isomorphic. We denote any of them
by $Z(k[G_l])$. Also we denote structure constants of $Z(k[G_l])$ by $S_{c'(l),c''(l)}^{c(l)}$

\begin{lemma} \label{l.P=S} Let $l$ be an orbit from $L=\Lambda/G$. Then a structure constant $P_{(l,c'),(l,c'')}^{(l,c)}$ of IK-algebra $A$ (resp., $A_{\preceq l}$) is equal to the structure constant $S_{c'(l),c''(l)}^{c(l)}$ of the center $Z(k[G_\lambda])$ of group algebra of the group $G_\lambda$.
\end{lemma}

\begin{proof} Lemma follows from the lemma \ref{l.struct_const} \end{proof}

Fix $\lambda \in \Lambda$.
 Let $A_{\preceq\lambda}$ be a generalized
 IK-algebra associated with the admissible set of data $(\Lambda_{\preceq\lambda},G_\lambda,\eta_\lambda)$. Then
by lemma \ref{l.P=S}, the linear subspace $A'$ generated by elements $\{e_{(\lambda,c)_\lambda}\}$, where $c$ run over conjugacy classes of $G$ intersecting with $G_\lambda$, is a subalgebra isomorphic to the center $Z(k[G_\lambda])$.

     For an element $h\in G_\lambda$ and $G$-orbit $l'$ in $\Lambda$ denote by $\xi( l',h; \lambda)$ the number of partial elements $(\lambda',h)$ such that $\lambda'\in l'$ and $\lambda'\preceq\lambda$.
    Clearly,
$\xi(l',h'; \lambda')=\xi(l',h; \lambda)$ if $h'$ is conjugated to $h$, $h'\in G_{\lambda'}$  and $\lambda'$ belongs to the orbit of $\lambda$. Thus, we may put $\xi(l',c;l)=\xi(l',h'; \lambda')$ where  $h'\in c$ and $\lambda\in l$.

\begin{lemma} [Main] \label{main_lemma} Let $l$, $l''$, $l''$ be three $G$-orbits in $\Lambda$ and $c$, $c'$, $c''$ be three conjugacy classes of the group $G$. Then
$$\xi(l',c';l)\xi(l'',c'';l) \, S_{c'(l),c''(l)}^{c(l)} =
\sum_{\tilde{l}}\xi( \tilde{l},c;  l)   P_{(l',c'),(l'',c'')}^{(\tilde{l},c)}
$$
\end{lemma}
\begin{proof} Fix $\lambda\in l$ and $h\in c(\lambda)$. Denote be $M$ the set of paires $( (\lambda',h'),(\lambda'',h''))$ such that
$h'h''=h$, $\lambda'\in l'$, $\lambda''\in l''$. Count the number of elements in $M$ in two ways.

First, count the number of
$\{h',h'' | h'h''=h\}$ and multiply it by two numbers:
$|\{\lambda' |\lambda'\in l', \lambda'\le\lambda,  G_{\lambda'}\ni h'\}|$
and  similar for $\lambda''$. We obtain left side of the identity.

     Second, group pairs of partial elements $(\lambda',h')$, $(\lambda'',h'')$ by their products $(\tilde{\lambda},h)$. Thus, in one group fall all pairs $(\lambda',h')$, $(\lambda'',h'')$ such that $\lambda'\wedge\lambda''=\tilde{\lambda}$.
     Multiply the number elements in each group $(\tilde{\lambda},h)$ by
the number of $\{\tilde{\lambda} | G_{\tilde{\lambda}}\ni h\}$.
     We obtain right side of the identity.
\end{proof}
     Lemma \ref{main_lemma} provides explicit expression of structure constants of the centers of group algebras of $G_\lambda$ via structure constants of the algebra of conjugacy classes of partial elements.
    This formula can be converted. Below we provide expression for $P_{\omega',\omega''}^\omega$ via $S_{c'(l),c''(l)}^{c(l)}$ in the special case:
poset $L=\Lambda/G$ is ordered set.
     Thus, $L$ may be identified with the set $\{0\}\cup\mathbb{N}$.

      In all examples from section \ref{examples} $L$ is ordered set.

     Fix conjugacy classes of partial elements  $\omega'=(l',c')$, $\omega''=(l'',c'')$ and conjugacy class $c$. To simplify formulas, below we use the following denotations:
$p^l=P_{\omega',\omega''}^{\omega}$ , where $\omega=(l,c)$ for an arbitrary $l\in L$;
$\xi'(l)=\xi(l',c';l)$, $\xi''(l)=\xi(l'',c'';l)$,
$s(l) =S_{c'(l),c''(l)}^{c(l)}$,
$\xi(\tilde{l},l)=\xi(\tilde{l},c;l)$.

Suppose, $l'\wedge l'' = \{m,m+1,\dots, M\}$. Therefore, $l$ run over the set $\{m,m+1,\dots, M\}$ and $\tilde{l}$
belong to the set $\{m,m+1,\dots, l\}$ because $\tilde{l}\preceq l$.

Define a vector $\overrightarrow{P}$ which components are structure constants of the algebra $A$:
$$
\overrightarrow{P} =  \left(
                      \begin{array}{c}
                                   p^m\\
                                   p^{m+1}\\
                                   p^{m+2}\\
                                  \dots\\
                                  p^M\\
                    \end{array}
                       \right)
$$
Define a matrix $R$  which coefficients are $\xi(\tilde{l},l)$:
$$
R= \left|  \begin{array}{ccccc}
                                0&                                   0&  0&  \dots&      0\\
                                \xi(m,m+1)&                    0&  0&  \dots&     0\\
                                \xi(m,m+2)&\xi(m+1,m+2)&  0&   \dots&     0\\
                                \dots&                       \dots&  \dots&   \dots&\dots\\
                                 \xi(m,M)&       \xi(m+1,M)& \xi(m+2,M)& \dots&        0\\
                \end{array}
      \right|
$$
Define a vector $ \overrightarrow{S}$ which components are structure constants of algebras $Z(k[G_l])$ multiplied by coefficients:
$$ \overrightarrow{S} =  \left(
                      \begin{array}{c}
                            \xi'(m)\xi''(m)s(m)\\
                            \xi'(m+1)\xi''(m+1)s(m+1)\\
                            \dots\\
                            \xi'(M)\xi''(M)s(M)\\
                    \end{array}
                       \right)
$$

The matrix $R$ is nilpotent, hence $1+R$ is unipotent and thus, invertible.

\begin{lemma}\label{inverse_lemma}
$$
\overrightarrow{P} = (1+R)^{-1} \overrightarrow{S}
$$
\end{lemma}

\begin{proof} Up to denotations, this lemma is equivalent to the lemma \ref{main_lemma}
\end{proof}

Constants $\xi(l',c;l)$ may be computed directly for admissible families of subgroups from section \ref{s.examples}. Moreover, they were computed in \cite{IK}, see also \cite{MMN2}, for the admissible set of subgroups of $G=S_\infty$.
In the case of the admissible set of subgroups of $F\rwr S_\infty$ they may be computed similarly.

The support of an elements $h$ of a group $G$ is defined in section \ref{s.examples} for
$G=F\rwr S_\infty$. In the case of $G=S_\infty$ define the support $\varsigma$ of $h$ as the set of $i\in\mathbb{N}$ such that $h(i)\ne i$.

Denote by $\varsigma$ the support of an element $h$ of a conjugacy class $c$.  Clearly, the cardinality
$\alpha=|\varsigma|$ is equal for all $h\in c$; we denote it by $\alpha(c)$.

\begin{proposition}
$\xi(l',c;l)= \frac{(l-\alpha(c))!}{(l-l')!(l'-\alpha(c))!}$
\end{proposition}

\begin{proof}Proof is by direct calculation. \end{proof}
\section{Monomorphism of generalized IK-algebra $A$ into direct sum of centers of group algebras $Z(k[G_l])$}
Let $\eta(\Lambda)\subset G$ be an admissible family of subgroups.

Denote by $\hat A$ the direct product of centers  $Z(k[G_l])$ of group algebras of groups $G_l$, $l\in \Lambda/G$.
Elements of $\hat A$ are (possibly, infinite) sums $\sum _{l\in L} a_l$, $a_l \in Z(k[G_l])$, the product of
elements $a'=\sum _{l\in L} a'_l$ and $a''=\sum _{l\in L} a''_l$ is $a'\circ a'' = \sum _{l\in L} a'_l\underset{l}{\circ} a''_l$ where $\underset{l}{\circ}$ denotes multiplication in corresponding algebra $Z(k[G_l])$.

Denote the sum of elements of a conjugacy class $c(l)\subset G_l$ by $e_{c(l)}$. Elements $e_{c(l)}$ form a basis of algebra $Z(k[G_l])$.

Define linear map $\varphi{:} A{\to} \hat{A}$ by the formula
$\varphi(e_{(l',c)}) = \sum_{l\succeq l'} \xi(l',c;l) e_{c(l)}$.

\begin{theorem} $\varphi$ is a monomorphism of algebras.
\end{theorem}
\begin{proof}
First, compute $l^{th}$ component $\lbrack \varphi(e_{(l',c')}){\circ}\varphi(e_{(l'',c'')})\rbrack_l$ of the product $\varphi(e_{(l',c')}){\circ}\varphi(e_{(l'',c'')})$ in the algebra $\hat A$:
\begin{equation*}
\begin{split}
&\lbrack \varphi(e_{(l',c')}){\circ}\varphi(e_{(l'',c'')})\rbrack_l=
\xi(l',c';l)\xi(l'',c'';l) e_{c'(l)}\underset{l}{\circ} e_{c''(l)}=\\
&=\xi(l',c';l)\xi(l'',c'';l)\sum_{c(l)}S_{c'(l),c''(l)}^{c(l)} e_{c(l)}
\end{split}
\end{equation*}
Second, compute $l^{th}$ component of  $\varphi(e_{ (l',c')} e_{(l'',c'')} )$:
 \begin{equation*}
\begin{split}
&\lbrack\varphi(e_{ (l',c')} e_{(l'',c'')} )\rbrack_l=
\lbrack\varphi( \sum_{(\tilde{l},c)} P_{(l',c'),(l'',c'')}^{(\tilde{l},c)}) e_{(\tilde{l},c)})\rbrack_l =\\
& =\sum_{(\tilde{l},c)} P_{(l',c'),(l'',c'')}^{(\tilde{l},c)})\xi(\tilde{l},c;l) e_{c(l)}
= \sum_{c(l)} \sum_{\tilde{l}} P_{(l',c'),(l'',c'')}^{(\tilde{l},c)})\xi(\tilde{l},c;l) e_{c(l)}
\end{split}
\end{equation*}

Thus, the statement follows from the lemma \ref{main_lemma}
\end{proof}

\section{Completion of IK-algebra $A$}

     Let $\eta(\Lambda)\subset G$ be an admissible family of subgroups and $A$ be generalized IK-algebra.
     For an orbit $l \in L=\Lambda/G$ denote by $A_l$ the linear subspace (not a
subalgebra!) of $A$ generated by basic vectors $e_{(l,c)}$ with fixed $l$ and arbitrary conjugacy classes $c$ of group $G$.
     Clearly, $A=\oplus_{l\in L} A_l$.
     Subspaces $A_l$ are finite dimensional because subgroups $G_\lambda$ of class $G_l$ are finite.

     Denote by $\bar A$ a linear space of formal sums $ \sum_{l\in L} a_l$  where $a_l\in A_l$. The product of two elements of $\bar A$,
$a=\sum_{l\in L} a_l$ and $b=\sum_{l\in L} b_l$, is defined correctly. Indeed,
$$
\sum_{l\in L} a_l\sum_{l\in L} b_l=\sum_{l',l''}a_{l'}b_{l''}=\sum_{l',l''}\sum_{l\in l'\wedge l''} c_{l',l'',l}
= \sum_{l}\sum_{l',l'': l'\wedge l''\ni l} c_{l',l'',l}
$$
Here $c_{l',l'',l}$ denotes the projection of the product $a(l')b(l'')$ to the component $A_l$.

Internal sum in the most right expression includes only finitely many summands because $l',l''\preceq l$ and there are only finitely many elements preceding $l$ in the poset $L$.
Thus, $ab= \sum_{l}d_l$  where $d_l=\sum_{l\in l'\wedge l''}c_{l',l'',l}$. Therefore, $\bar A$ is an algebra.

If  $|L|\le\infty$, then evidently, $\bar A = A$. In the case of infinite set $L$ algebra $\bar A$ is a completion of $A$ in the following topology. We call a finite subset $F\subset L$ closed if $l\preceq f\in F$ implies $l\in F$. For a closed finite subset $F\subset  L$ denote a linear subspace $\oplus_{l\notin F}A_l$ of the algebra $A$  by $A_F$.
Call sets $\{A_F| F \subset L, |F|<\infty\}$  a fundamental system of neighborhoods of zero.  Clearly, $\bigcup_{F}A_F=A$ and
$\bigcap_{F}A_F = \{0\}$.  Evidently,  $A$ is dense in $\bar A$.

Define topology on the algebra $\hat A=\prod\limits_{l\in L} Z(k[G_l])$ similarly.
\begin{lemma} Homomorphism $\varphi: A\to \hat A$ is continuous.
\end{lemma}
\begin{proof}
For any neighborhood $\hat A_F\subset \bar A$  corresponding to a finite subset $F\subset L$  the image of the neighborhood $A_{{}^\wedge F}$ is
evidently contained in $\hat A_F\subset \bar A$. Thus, $\varphi$ is continuous.
\end{proof}

Denote by $\bar \varphi$ the extension of the homomorphism $\varphi:A\to \hat A$
to the algebra $\bar A$ by continuity.

\begin{theorem} Homomorphism $\bar \varphi:\bar A\to \hat A$ is an isomorphism of algebras.
\end{theorem}

\begin{proof}
    Homomorphism $\bar \varphi$ is a monomorphism  because $\varphi$ is monomorphism. To prove that the image $\bar\varphi(\bar A)$ is dense in $\hat A$, let us choose any basic element $e_{(l,c)}\in A$.
    By definition, the image of it is $\varphi(e_{(l,c)})=\sum_{l'\succeq l}\xi(l,c;l') e_{c(l')}$. We shell prove that the first summand of this row belongs to the closure of $\varphi(A)$. Take $l'\succeq l$ such that if
$x\in L$ and $l\preceq x\preceq l'$  then either $x=l$ or $x=l'$.
    The image of $e_{(l',c)}$ is $\sum_{l''\succeq l'}\xi(l',c;l'')e_{c(l'')}$.
    Choosing appropriate coefficient $\beta_{l'}$ we get that $l'$-th component of $\varphi(e_{(l,c)})-\beta_{l'}\varphi(e_{(l',c)})$ is zero.
    Continuing inductively this procedure, we obtain the row
 $\varphi(e_{(l,c)})-\sum_{l'\succ l}\beta_{l'}\varphi(e_{(l',c)})$ that converges to $e_{c(l)}$.
    The induction is valid here because for each $l'\in L$ there is finitely many $l\in L$  such that $l\preceq l'$.
\end{proof}

\section*{Acknowledgments}

The work of second author was supported, in part,  by Ministry of Education and Science of the Russian Federation under contract 8498,
Russian Federation Government Grant No. 2010-220-01-077, ag.no.11.G34.31.0005, NSh-4850.2012.1, RFBR grants 11-01-00289.  The study of second author was carried within "The National Research University Higher School of Economics" Academic Fund Program in 2013-2014, research grant No. 12-01-0122.

A. Alekseevski

Belozersky inst. of Moscow State University, Leninskie Gory 1-40, Moscow 119991, Russia

Scientific Research Institute for System Studies (NIISI RAN), Moscow, Russia

aba@belozersky.msu.ru

S.Natanzon

National Research University Higher School of Economics, Moscow Vavilova 7, Russia

Belozersky inst. of Moscow State University,
Leninskie Gory 1-40, Moscow 119991, Russia

Institute for Theoretical and Experimental Physics, Moscow, Russia

natanzons@mail.ru


\begin{thebibliography}{100}

\bibitem{AN} Alexeevski A., Natanzon S., Noncommutative  two-dimensional
topological field theories and Hurwitz numbers for real algebraic
curves. Selecta Math., New ser. v.12,n.3, 2006, p. 307-377
(arXiv: math.GT/0202164).

\bibitem{AN1} Alexeevski A., Natanzon S., Algebra of Hurwitz numbers for
seamed surfaces, Russian Math.Surveys, 61 (4) (2006), 767-769

\bibitem{AN2} Alexeevski A., Natanzon S., Algebra of bipartite graphs and
Hurwitz numbers of seamed surfaces. Math.Russian
Izvestiya 72 (2008) V.4, 3-24.

\bibitem{Carter} Carter R.W., Conjugacy Classes in the Weyl Group,
Compositio Mathematica vol. 25 (1972), Fasc.1 , 1-59.

\bibitem{D1} Dijkgraaf  R., Geometrical Approach to Two-Dimensional
Conformal Field Theory, Ph.D.Thesis (Utrecht, 1989)

\bibitem{D2} Dijkgraaf  R., Mirror symmetry and elliptic curves,
The moduli spaces of curves, Progress in Math., 129 (1995), 149-163,
Birkh\"auser.

\bibitem{IK} Ivanov V., Kerov S., The Algebra of Conjugacy Classes in Symmetric Groups and Partial Permutations,
Journal of Mathematical Sciences (Kluwer) 107 (2001) 4212-4230 (arXiv: math/0302203).

\bibitem{MMN1} Mironov A., Morozov A., Natanzon S.,
Complete Set of Cut-and-Join Operators in Hurwitz-Kontsevich Theory,
Theor.Math.Phys. 166 (2011) 1-22 (arXiv:0904.4227).

\bibitem{MMN2} Mironov A., Morozov A., Natanzon S.,
Algebra of differential operators associated with Young diagrams,
Journal of Geometry and Physics 62 (2012) 148-155 (arXiv:1012.0433).

\bibitem{MMN3} Mironov A., Morozov A., Natanzon S.,
Cardy-Frobenius extension of algebra of cut-and-join operators,
Journal of Geometry and Physics 73 (2013) 243-251
arXiv:1210.6955.

\bibitem{MMN4} Mironov A., Morozov A., Natanzon S.,
Asymptotic Hurwitz numbers, arXiv:1212.2041.

\bibitem{OO} Okounkov A, Olshanski G, Shifted Schur functions,
ST.Petersburg Math. J. 9(1998), 2  arXiv:q-alg/9605042.


\end{thebibliography}
\end{document}